\documentclass[12pt]{amsart}
\usepackage{ben}

\title[Regularity of solutions to NS with $\bmo^{-1}$ initial data]{On regularity of solutions to the Navier--Stokes equation with initial data in $\bmo^{-1}$}
\author{Hedong Hou}
\address{Universit{\'e} Paris-Saclay, CNRS, Laboratoire de Math\'{e}matiques d'Orsay, 91405 Orsay, France}
\email{hedong.hou@universite-paris-saclay.fr}

\date{December 18, 2024}

\begin{document}

\begin{abstract}
    We prove that any mild solution in the Koch--Tataru space to the incompressible Navier--Stokes equation with initial data in $\bmo^{-1}$ is weak*-continuous in time, valued in $\bmo^{-1}$. We also show that the global mild solution vanishes in $\bmo^{-1}$ at infinity in time.
\end{abstract}
\maketitle

\section{Introduction}
\label{sec:intro}

This paper is concerned with regularity of mild solutions to the Cauchy problem of the incompressible Navier--Stokes equation
\begin{equation}
    \tag{NS}
    \label{e:NS}
    \begin{cases}
        \partial_t u + (u \cdot \nabla) u = \Delta u - \nabla p, \\
        \nabla \cdot u = 0, \\
        u(0)=u_0, \quad \nabla \cdot u_0 = 0,
    \end{cases}
\end{equation}
where $u:[0,\infty) \times \bR^n \to \bR^n$ denotes the unknown velocity vector, and $p:[0,\infty) \times \bR^n \to \bR$ denotes the unknown scalar pressure. Let $0<T \le \infty$. For a divergence-free tempered distribution $u_0$, we say $u$ is a \emph{mild} solution to \eqref{e:NS} with initial data $u_0$ if it satisfies the integral equation
\[ u(t) = e^{t\Delta} u_0 - B(u,u)(t) \]
in the sense of distributions $\scrD'((0,T) \times \bR^n)$,
\footnote{We denote spaces of scalar-valued and spaces of vector-valued functions or distributions in the same way.}
where the bilinear operator $B$ is formally given by the integral
\begin{equation}
    \label{e:B}
    B(u,u)(t) := \int_0^t e^{(t-s)\Delta} \bP \Div (u \otimes u)(s) ds, \quad t \in (0,T).
\end{equation}
Here, $\bP$ denotes the Leray projection on divergence-free vector fields. 

Based on the scaling property of the equation
\[ u(t,x) \mapsto \lambda u(\lambda^2 t,\lambda x), \]
there is a chain of scale-invariant (or called \emph{critical}) spaces
\[ \DotH^{\frac{n}{2}-1} \hookrightarrow L^n \hookrightarrow \DotB^{-1+\frac{n}{p}}_{p,\infty} \hookrightarrow \bmo^{-1} \hookrightarrow \DotB^{-1}_{\infty,\infty}, \quad 2 \le n<p<\infty. \]
The pioneering work of Fujita--Kato \cite{Fujita-Kato1964_Hn/2-1} establishes global existence of mild solutions $u \in C([0,\infty);\DotH^{n/2-1})$, provided that the initial data $u_0$ is small. Such existence results are also valid in $L^n$ due to Kato \cite{Kato1984-Ln} and in $\DotB^{-1+n/p}_{p,\infty}$ (with weak*-topology) due to Cannone \cite{Cannone1997-Besov-sge} and Planchon \cite{Planchon1998-Besov-sgwp}. On the other hand, the work of Bourgain--Pavlovi\'c \cite{Bourgain-Pavlovic2008-IPB-1} addresses ill-posedness of \eqref{e:NS} in $\DotB^{-1}_{\infty,\infty}$. See also the works of Yoneda \cite{Yoneda2010-IP-F-1_infty-q} and Wang \cite{Wang2015-IP-B-1_infty-q}.

For $\bmo^{-1}$, Koch--Tataru \cite{Koch-Tataru2001-BMO-1} establish global existence of mild solutions $u \in X_\infty$ (see Section \ref{sec:spaces} for definition) for small initial data. Auscher--Dubois--Tchamitchian \cite{Auscher-Dubois-Tchamitchian2004-BMO-1} further show that the solution $u$ belongs to $L^\infty((0,\infty);\bmo^{-1})$. Miura--Sawada \cite{Miura-Sawada2006_analyticity} and Germain--Pavlovi\'c--Staffilani \cite{Germain-Pavlovic-Staffilani2007-Analyticity} obtain spatial analyticity of $u$. But time regularity remains unknown.

The main goal of this paper is to address this issue with the following result.

\begin{theorem}
    \label{thm:cont-bmo-1}
    Let $0<T \le \infty$. Let $u_0 \in \bmo^{-1}$ be divergence free and $u \in X_T$ be a mild solution to the Navier--Stokes equation \eqref{e:NS} with initial data $u_0$. Then $u$ belongs to $C([0,T);\bmo^{-1})$ with $u(0)=u_0$ and
    \begin{equation}
        \label{e:bd-u-C(bmo-1)}
        \sup_{0 \le t < T} \|u(t)\|_{\bmo^{-1}} \lesssim \|u_0\|_{\bmo^{-1}} + \|u\|_{X_T}^2.
    \end{equation}
\end{theorem}

Here and in the sequel, $\bmo^{-1}$ is equipped with the weak*-topology with respect to the homogeneous Hardy--Sobolev space $\DotH^{1,1}$, see Section \ref{sec:spaces} for details.

The next result is concerned with the long-time behavior of global mild solutions. We refer the reader to \cite{Gallagher-Iftimie-Planchon2002-long-H1/2} for motivation to study the asymptotic behavior of global mild solutions, and to \cite{Planchon1998-Besov-sgwp,Gallagher-Iftimie-Planchon2003-long-L3-Besov} for results in other critical spaces.
\begin{theorem}
    \label{thm:long-time}
    Let $u_0 \in \bmo^{-1}$ be divergence free and $u \in X_\infty$ be a global mild solution to the Navier--Stokes equation \eqref{e:NS} with initial data $u_0$. Then $u(t)$ converges to 0 as $t \to \infty$ in $\bmo^{-1}$ (also endowed with the weak*-topology).
\end{theorem}

Let us briefly discuss the earlier works. Koch--Tataru \cite{Koch-Tataru2001-BMO-1} establish local well-posedness of \eqref{e:NS} for initial data in the closure of Schwartz functions in $\bmo^{-1}$, denoted by $\vmo^{-1}$. More precisely, for any divergence-free $u_0 \in \vmo^{-1}$, there exist a time $T>0$ and a mild solution $u \in X_T$ to \eqref{e:NS} with initial data $u_0$, which also satisfies
\begin{equation}
    \label{e:local-small-VMO-1}
    \lim_{\tau \to 0+} \|u\|_{X_\tau} = 0.
\end{equation}
Dubois \cite{Dubois2002-these} proves that any mild solution $u \in X_T$ with initial data $u_0 \in \vmo^{-1}$ is strongly continuous in time, valued in $\bmo^{-1}$. 

Next, Dubois \cite{Dubois2002-these} and Auscher--Dubois--Tchamitchian \cite{Auscher-Dubois-Tchamitchian2004-BMO-1} present the long-time limit of $u$, namely, any global mild solution $u \in X_\infty$ with initial data $u_0 \in \vmo^{-1}$ satisfies
\begin{equation}
    \label{e:long-VMO-1}
    \lim_{t \to \infty} \sqrt{t} \|u(t)\|_{L^\infty} = 0, \quad \lim_{t \to \infty} \|u(t)\|_{\bmo^{-1}} = 0.
\end{equation}
See also the work of Germain \cite{Germain2006-2DVMO-1-glexist-long} for a particular 2D case. 

Therefore, Theorems \ref{thm:cont-bmo-1} and \ref{thm:long-time} could be understood as the complement of \cite{Dubois2002-these,Auscher-Dubois-Tchamitchian2004-BMO-1} without local smallness \eqref{e:local-small-VMO-1}.

\begin{remark}
    \label{rem:sharp}
    In fact, Theorems \ref{thm:cont-bmo-1} and \ref{thm:long-time} are sharp. 
    \begin{enumerate}[label=(\roman*)]
        \item \label{item:sharp_cont}
        The heat semigroup $(e^{t\Delta})$ is merely weak*-continuous on $\bmo^{-1}$, so the weak*-continuity of mild solutions with $\bmo^{-1}$-initial data, as asserted in Theorem \ref{thm:cont-bmo-1}, might possibly be the best expected. In comparison, $(e^{t\Delta})$ is strongly continuous in $\vmo^{-1}$, which aligns with the strong continuity of mild solutions with $\vmo^{-1}$-initial data, as obtained by Dubois.

        \item \label{item:sharp_long}
        The long-time behavior in Theorem \ref{thm:long-time} fails if we consider the strong topology of $\bmo^{-1}$. Indeed, let $u_0 \in \bmo^{-1}$ be non-zero, divergence free, and homogeneous of degree $-1$, with small $\bmo^{-1}$-norm. Then there exists a unique small self-similar global mild solution $u \in X_\infty$ to \eqref{e:NS} with initial data $u_0$, i.e.,
        \[ u(t,x) = \frac{1}{\sqrt{t}} U \left( \frac{x}{\sqrt{t}} \right), \quad (t,x) \in (0,\infty) \times \bR^n, \]
        for some non-zero divergence-free $U \in C^\infty(\bR^n)$, see \cite[Theorem 2.7]{Germain-Pavlovic-Staffilani2007-Analyticity}. In particular, by scale invariance, one gets that for such self-similar $u$,
        \[ \|u(t)\|_{\bmo^{-1}} = \|U\|_{\bmo^{-1}} \ne 0, \quad t>0. \]
        Therefore, it never vanishes in the strong topology.

        This argument does not conflict with the results by Auscher, Dubois, and Tchamitchian for $\vmo^{-1}$-initial data, cf. \eqref{e:long-VMO-1}, as we notice that there is no non-zero distribution in $\vmo^{-1}$ that is homogeneous of degree $-1$.
    \end{enumerate}
\end{remark}

To finish the introduction, let us mention the problem of uniqueness. To our best knowledge, uniqueness of mild solutions for arbitrary initial data in $\bmo^{-1}$ still remains open. Some partial results require additional smallness (e.g. \cite{Koch-Tataru2001-BMO-1}) or local smallness (e.g. \cite{Dubois2002-these,Miura2005-unique-vmo-1}) of the solution. For instance, Dubois \cite{Dubois2002-these} proves that any mild solution $u \in X_T$ satisfying \eqref{e:local-small-VMO-1} with $u_0 \in \bmo^{-1}$ (in particular for $u_0 \in \vmo^{-1}$) is unique. However, we also notice that a recent paper of Guillod--{\v S}ver\'ak \cite{Guillod-Sverak2023-nonunique} provides numerical evidence suggesting non-uniqueness for large initial data in $\bmo^{-1}$ by symmetry breaking.

\subsection*{Organization}
In Section \ref{sec:spaces}, we introduce the function spaces and their basic properties to be used. Section \ref{sec:outline} outlines the proofs of Theorems \ref{thm:cont-bmo-1} and \ref{thm:long-time}, with two main propositions (cf. Propositions \ref{prop:L0} and \ref{prop:L1}) proved in Sections \ref{sec:L0} and \ref{sec:L1}, respectively.

\subsection*{Notation}
\label{ssec:notation}
Throughout the paper, we say $X \lesssim Y$ (or $X \lesssim_A Y$, resp.) if $X \le CY$ with an irrelevant constant $C$ (or depending on $A$, resp.), and say $X \eqsim Y$ if $X \lesssim Y$ and $Y \lesssim X$. 

Let $n \ge 1$ and write $\bR^{1+n}_+:=(0,\infty) \times \bR^n$. For any function $f(t,x)$ on $(0,\infty) \times \bR^n$, denote by $f(t)$ the function $x \mapsto f(t,x)$. For any (Euclidean) ball $B \subset \bR^n$, write $r(B)$ for the radius of $B$. Let $E$ be a measurable subset of $\bR^n$ with finite positive measure and $f \in L^1(E)$. We write
\[ \fint_E f dx := \frac{1}{|E|} \int_E f dx, \]
where $|E|$ denotes the Lebesgue measure of $E$. 

\subsection*{Acknowledgements}
The author is very grateful to his supervisor Pascal Auscher for his generous help on the proofs and the writing. The author also thanks Patrick G\'erard and Zachary Bradshaw for enlightening discussions.
\section{Function spaces}
\label{sec:spaces}

Let $\bmo^{-1}$ be the collection of distributions $f \in \scrD'(\bR^n)$ with $f=\Div g$ for some $g \in \bmo$. The space $\bmo^{-1}$ is isomorphic to the homogeneous Triebel--Lizorkin space $\DotF^{-1}_{\infty,2}$ and the homogeneous Hardy--Sobolev space $\DotH^{-1,\infty}$. The reader can refer to \cite[\S 5]{Triebel1983Spaces} for definitions and fundamental properties of the family of homogeneous Hardy--Sobolev spaces $\DotH^{s,p}$ with $s \in \bR$ and $0<p\le \infty$, or \cite[\S 2.1]{Auscher-Hou2024-HCL} for a short summary. In particular, let us mention two basic properties to be frequently used.
\begin{itemize}
    \item (Duality) The space $\bmo^{-1}$ identifies with the dual of $\DotH^{1,1}$ via $L^2(\bR^n)$-duality.

    \item (Density) Let $\scrS_\infty$ be the subspace of Schwartz functions $\scrS$ consisting of $\phi \in \scrS$ with $\int_{\bR^n} x^\alpha \phi(x) dx = 0$ for any multi-index $\alpha$. It is dense in $\DotH^{s,p}$ for any $s \in \bR$ and $0<p<\infty$.
\end{itemize}

For $1 \le q < \infty$ and $0<T \le \infty$, define the \emph{Carleson functional} $\cC^{(q)}_T$ on (vector-valued, strongly) measurable functions $u$ on $(0,T) \times \bR^n$ by
\[ \cC^{(q)}_T(u)(x) := \sup_{\substack{B:x \in B \\ r(B)^2<T}} \left( \int_0^{r(B)^2} \fint_B |u(t,y)|^q dtdy \right)^{1/q}, \quad x \in \bR^n, \]
where $B$ describes balls in $\bR^n$. Let $X_T$ be the collection of measurable functions $u$ on $(0,T) \times \bR^n$ for which
\[ \|u\|_{X_T} := \sup_{0<t<T} \|t^{1/2} u(t)\|_{L^\infty} + \|\cC^{(2)}_T(u)\|_{L^\infty} < \infty. \]

To have a better understanding of the norm $\|\cC^{(q)}_T(u)\|_{L^\infty}$, we recall tent spaces introduced by \cite{Coifman-Meyer-Stein1985_TentSpaces}. Here we adapt to the parabolic settings. Let $\beta \in \bR$ and $1 \le q < \infty$. The \emph{tent space $T^{\infty,q}_\beta$} consists of measurable functions $u$ on $\bR^{1+n}_+$ for which
\[ \|u\|_{T^{\infty,q}_\beta} := \sup_B \left( \int_0^{r(B)^2} \fint_B |t^{-\beta} u(t,y)|^q dtdy \right)^{1/q} < \infty. \]
For coherence, we also write
\[ \|u\|_{L^\infty_\beta(\bR^{1+n}_+)} := \sup_{t>0} \|t^{-\beta} u(t)\|_{L^\infty}. \]
Observe that $\|u\|_{T^{\infty,q}_\beta} = \|\cC^{(q)}_\infty(t^{-\beta} u(t))\|_{L^\infty}$. In fact, most of the properties of tent spaces apply \textit{mutatis mutandis} to the norm $\|\cC^{(q)}_T(t^{-\beta} u(t))\|_{L^\infty}$ for $T<\infty$. We shall omit to mention this in the sequel. 

For $\beta \in \bR$ and $1<q<\infty$, $T^{\infty,q}_\beta$ identifies with the dual of $T^{1,q'}_{-\beta}$ via $L^2(\bR^{1+n}_+)$-duality, where the \emph{tent space $T^{1,q'}_{-\beta}$} consists of measurable functions $u$ on $\bR^{1+n}_+$ for which
\[ \|u\|_{T^{1,q'}_{-\beta}} := \int_{\bR^n} \left( \int_0^\infty \fint_{B(x,t^{1/2})} |t^{\beta} u(t,y)|^{q'} dtdy \right)^{1/q'} dx < \infty. \]
For $q=1$, the duality is more subtle (see \cite[Proposition 1]{Coifman-Meyer-Stein1985_TentSpaces}) but for any measurable functions $u,v$ on $\bR^{1+n}_+$, it still holds that
\begin{equation}
    \label{e:dual-T^infty,1}
    \int_0^\infty \int_{\bR^n} |u(t,y)| |v(t,y)| dtdy \lesssim \|u\|_{T^{\infty,1}_\beta} \|v\|_{T^{1,\infty}_{-\beta}},
\end{equation}
where the $T^{1,\infty}_{-\beta}$-norm is given by
\[ \|v\|_{T^{1,\infty}_{-\beta}} := \int_{\bR^n} \left( \esssup{\substack{t>0, |y-x|<t^{1/2}}} |t^\beta v(t,y)| \right) dx. \]

A well-known heat characterization of $\bmo^{-1}$ provides the link between $\bmo^{-1}$ and the tent space $T^{\infty,2}$ as
\[ \|u_0\|_{\bmo^{-1}} \eqsim \|e^{t\Delta} u_0\|_{T^{\infty,2}}. \]
Here and in the sequel, we omit the script $\beta$ for tent spaces if $\beta=0$.
\section{Outline of the proofs}
\label{sec:outline}

In this section, we outline the proofs of Theorems \ref{thm:cont-bmo-1} and \ref{thm:long-time}. It is enough to prove the case $T=\infty$. The same argument also works when $T$ is finite. Let $u \in X_\infty$ be a (global) mild solution to \eqref{e:NS} with initial data $u_0 \in \bmo^{-1}$. H\"older's inequality yields
\[ u \otimes u \in L^\infty_{-1}(\bR^{1+n}_+) \cap T^{\infty,2}_{-1/2} \cap T^{\infty,1}. \]
Define the linear operator
\[ \cL(f)(t) := \int_0^t e^{(t-s)\Delta} \bP \Div f(s) ds, \quad t>0. \]
Theorems \ref{thm:cont-bmo-1} and \ref{thm:long-time} follow from the following main proposition.
\begin{prop}
    \label{prop:L}
    Let $f \in L^\infty_{-1}(\bR^{1+n}_+) \cap T^{\infty,2}_{-1/2} \cap T^{\infty,1}$. Then $\cL(f)$ belongs to $C_0([0,\infty);\bmo^{-1})$ with $\cL(f)(0)=0$.
\end{prop}
Here, $C_0([0,\infty);\bmo^{-1})$ is the space of continuous functions valued in $\bmo^{-1}$ (endowed with the weak*-topology), with limit 0 as $t \to \infty$.

Admitting Proposition \ref{prop:L}, let us prove Theorems \ref{thm:cont-bmo-1} and \ref{thm:long-time}.

\begin{proof}[Proof of Theorems \ref{thm:cont-bmo-1} and \ref{thm:long-time}, admitting Proposition \ref{prop:L}]
    Note that the heat semigroup $(e^{t\Delta})$ is bounded and weak*-continuous on $\bmo^{-1}$, one can find the function $t \mapsto e^{t\Delta} u_0$ lies in $C_0([0,\infty);\bmo^{-1})$ with
    \begin{equation}
        \label{e:heat-u_0-bmo-1}
        \sup_{t \ge 0} \|e^{t\Delta} u_0\|_{\bmo^{-1}} \eqsim \|u_0\|_{\bmo^{-1}}.
    \end{equation}
    Moreover, Proposition \ref{prop:L} yields $B(u,u) = \cL(u \otimes u)$ also belongs to $C_0([0,\infty);\bmo^{-1})$ with $B(u,u)(0)=0$. We hence obtain that $u$ lies in $C_0([0,\infty);\bmo^{-1})$ with $u(0)=u_0$.
    
    To show the bound \eqref{e:bd-u-C(bmo-1)}, we infer from \cite[Lemma 8]{Auscher-Dubois-Tchamitchian2004-BMO-1} that
    \begin{equation}
        \label{e:B(u,u)-Linfty-BMO-1}
        \sup_{t>0} \|B(u,u)(t)\|_{\bmo^{-1}} \lesssim \|u\|_{X_\infty}^2.
    \end{equation}
    In fact, this inequality extends to $t \ge 0$ due to the fact that $B(u,u)(t)$ tends to 0 in $\bmo^{-1}$ as $t \to 0+$. Combining this with \eqref{e:heat-u_0-bmo-1} gives the bound \eqref{e:bd-u-C(bmo-1)}. This completes the proof.
\end{proof}

Now we concentrate on the proof of Proposition \ref{prop:L}. As we shall see, our proof also implies the bound \eqref{e:B(u,u)-Linfty-BMO-1}, see Remark \ref{rem:B(u,u)-Linfty-BMO-1}.

\begin{proof}[Proof of Proposition \ref{prop:L}]
    We follow \cite{Auscher-Frey2017-NS} to decompose $\cL$ as
    \begin{align*}
        \cL(f)(t)
        &= \int_0^t \Delta e^{(t-s)\Delta} (s\Delta)^{-1} (\bI-e^{2s\Delta}) s^{1/2} \bP \Div s^{1/2} f(s) ds \\
        &\quad + \int_0^t e^{(t+s)\Delta} \bP \Div f(s) ds.
    \end{align*}

    As $f \in L^\infty_{-1}(\bR^{1+n}_+) \cap T^{\infty,2}_{-1/2}$, we get $s^{1/2}f(s) \in L^\infty_{-1/2}(\bR^{1+n}_+) \cap T^{\infty,2}$. Define the operator
    \[ \cR(F)(s) := (s\Delta)^{-1} (\bI-e^{2s\Delta}) s^{1/2} \bP \Div F(s), \quad s>0. \]
    \begin{lemma}
        \label{lemma:R}
        $\cR$ is bounded on both $L^\infty_{-1/2}(\bR^{1+n}_+)$ and $T^{\infty,2}$.
    \end{lemma}
    We defer the proof to the end of this section. Now we have
    \[ g(s) := (s\Delta)^{-1} (\bI-e^{2s\Delta}) s^{1/2} \bP \Div s^{1/2} f(s) \]
    belongs to $L^\infty_{-1/2}(\bR^{1+n}_+) \cap T^{\infty,2}$. Define the \emph{maximal regularity operator}
    \[ \cL_0(g)(t) := \int_0^t \Delta e^{(t-s)\Delta} g(s) ds, \quad t>0, \]
    and the remainder
    \[ \cL_1(f)(t) := \int_0^t e^{(t+s)\Delta} \bP \Div f(s) ds, \quad t>0. \]
    Proposition \ref{prop:L} directly follows from the following two propositions.
    \begin{prop}
        \label{prop:L0}
        Let $g \in L^\infty_{-1/2}(\bR^{1+n}_+) \cap T^{\infty,2}$. Then $\cL_0(g)$ lies in $C_0([0,\infty);\bmo^{-1})$ with $\cL_0(g)(0)=0$.
    \end{prop}

    \begin{prop}
        \label{prop:L1}
        Let $f \in T^{\infty,1}$. Then $\cL_1(f)$ lies in $C_0([0,\infty);\bmo^{-1})$ with $\cL_1(f)(0)=0$.
    \end{prop}

    The proofs are provided in Sections \ref{sec:L0} and \ref{sec:L1}, respectively.
\end{proof}

To end this section, let us prove Lemma \ref{lemma:R}.

\begin{proof}[Proof of Lemma \ref{lemma:R}]
    For $s>0$, define the operator $\cR_s$ on $L^2(\bR^n)$ as
    \[ \cR_s(h) := (s\Delta)^{-1} (\bI-e^{2s\Delta}) s^{1/2} \bP \Div h. \]
    Using Fourier transform, one gets that the family $(\cR_s)_{s>0}$ is uniformly bounded on $L^2(\bR^n)$. Denote by $K_s(x)$ the convolution kernel of $\cR_s$. We infer from singular integral realization of pseudo-differential operators (see e.g. \cite[\S VI.4]{Stein1993-HA}) that the kernel satisfies the pointwise bound
    \begin{equation}
        \label{e:K_s(x)}
        |K_s(x)| \lesssim s^{-n/2} \min\{ (s^{-1/2} |x|)^{-n+1}, (s^{-1/2} |x|)^{-n-1} \}
    \end{equation}
    for any $s>0$ and $x \in \bR^n \setminus \{0\}$. In particular, it lies in $L^1(\bR^n)$ with a bound independent of $s$. Thus, $(\cR_s)$ is uniformly bounded on $L^\infty(\bR^n)$, and we hence get boundedness of $\cR$ on $L^\infty_{-1/2}(\bR^{1+n}_+)$ as
    \begin{align*}
        \|\cR(F)\|_{L^\infty_{-1/2}(\bR^{1+n}_+)}
        &= \sup_{s>0} \|s^{1/2} \cR_s(F(s))\|_{L^\infty} \\
        &\lesssim \sup_{s>0} \|s^{1/2} F(s)\|_{L^\infty} = \|F\|_{L^\infty_{-1/2}(\bR^{1+n}_+)}.
    \end{align*}
    
    Next, boundedness of $\cR$ on $T^{\infty,2}$ follows from \cite[Lemma 3.1]{Auscher-Frey2017-NS} with slight modifications of the proof, once we show that for any $s>0$, $E,F \subset \bR^n$ as Borel sets with $\dist(E,F) \ge s^{1/2}$, and $h \in L^2$,
    \begin{equation}
        \label{e:L2-Linfty-ode-Rs}
        \|\I_F \cR_s (\I_E h)\|_{L^\infty} \lesssim s^{-\frac{n}{4}} \left(s^{-1} \dist(E,F)^2\right)^{-\frac{1}{4}(n+1)} \|\I_E h\|_{L^2}.
    \end{equation}
    
    To verify \eqref{e:L2-Linfty-ode-Rs}, we notice that the kernel bound \eqref{e:K_s(x)} implies 
    \[ \|\I_F \cR_s (\I_E h)\|_{L^\infty} \lesssim s^{-\frac{n}{2}} \left(s^{-1} \dist(E,F)^2\right)^{-\frac{1}{2}(n+1)} \|\I_E h\|_{L^1}. \]
    Since $(\cR_s)$ is uniformly bounded on $L^\infty$, we get \eqref{e:L2-Linfty-ode-Rs} by interpolation. This completes the proof.
\end{proof}

In a nutshell, it only remains to prove Propositions \ref{prop:L0} and \ref{prop:L1}. The proofs are presented in the following two sections.

\section{Proof of Proposition \ref{prop:L0}}
\label{sec:L0}

This section is concerned with the proof of Proposition \ref{prop:L0}. We first treat the continuity at $t=0$ and then at $t>0$. Finally we show the limit as $t \to \infty$. Recall that $\bmo^{-1}$ is equipped with the weak*-topology with respect to $\DotH^{1,1}$.

\subsection{Continuity at \texorpdfstring{$t=0$}{t=0}}
\label{ssec:L0-t=0}
To prove $\cL_0(g)(t)$ tends to 0 in $\bmo^{-1}$ as $t \to 0+$, we argue by duality. Pick $\varphi \in \scrS_\infty$. Fubini's theorem yields that for any $t>0$,
\[ \langle \cL_0(g)(t),\varphi \rangle_{L^2(\bR^n)} = \int_0^t \int_{\bR^n} g(s,y) e^{(t-s)\Delta} \Delta \varphi(y) dsdy. \]
It suffices to verify that
\begin{equation}
    \label{e:Phi->0}
    \Phi := \int_0^t \int_{\bR^n} |g(s,y)| |e^{(t-s)\Delta} \Delta \varphi(y)| dsdy \xrightarrow[t \to 0+]{} 0.
\end{equation}

To achieve this, we split $\Phi$ into two parts, the main term
\[ \Phi_1 := \int_0^{t/2} \int_{\bR^n} |g(s,y)| |e^{(t-s)\Delta} \Delta \varphi(y)| dsdy, \]
and the remainder for iteration
\[ R_1 := \int_{t/2}^t \int_{\bR^n} |g(s,y)| |e^{(t-s)\Delta} \Delta \varphi(y)| dsdy. \]

\subsubsection{Estimates of $\Phi_1$}
\label{sssec:Phi_1}
Thanks to duality of tent spaces (see \cite[Theorem 1(a)]{Coifman-Meyer-Stein1985_TentSpaces}), we have
\begin{equation}
    \label{e:Phi_1-dual-bd}
    \Phi_1 \lesssim \|\cC^{(2)}_{t/2}(g)\|_{L^\infty} \cA_t(\varphi),
\end{equation}
where the functional $\cA_t(\varphi)$ is given by
\[ \cA_t(\varphi) := \int_{\bR^n} \left( \int_0^{t/2} \fint_{B(x,s^{1/2})} |e^{(t-s)\Delta} \Delta \varphi(y)|^2 dsdy \right)^{1/2} dx. \]

The following lemma provides estimates of $\cA_t$.
\begin{lemma}
    \label{lemma:A_t}
    Let $\varphi \in \DotH^{1,1}$. Then
    \begin{equation}
        \label{e:A_t(phi)-bd}
        \sup_{t>0} \cA_t(\varphi) \lesssim \|\varphi\|_{\DotH^{1,1}}.
    \end{equation}
    Moreover,
    \begin{equation}
        \label{e:A_t(phi)-limit}
        \lim_{t \to 0+} \cA_t(\varphi) = \lim_{t \to \infty} \cA_t(\varphi) = 0.
    \end{equation}
\end{lemma}

We first prepare a technical lemma. Denote by $G_t(x)$ the heat kernel
\[ G_t(x) := \frac{1}{(4\pi t)^{n/2}} e^{-\frac{|x|^2}{4t}}. \]

\begin{lemma}
    \label{lemma:G*psi*a}
    Let $\psi \in C^\infty(\bR^n)$ be so that $\psi_t(x) := t^{-n/2} \psi (t^{-1/2} x)$ satisfies pointwise Gaussian estimates, i.e.,
    \begin{equation}
        \label{e:psi_t(x)}
        |\psi_t(x)| \lesssim \frac{1}{(4\pi t)^{n/2}} e^{-\frac{|x|^2}{4t}} = G_t(x), \quad t>0, ~ x \in \bR^n.
    \end{equation}
    Then for any $a \in \DotH^{0,1}$,
    \begin{equation}
        \label{e:G*psi*a-bd}
        \sup_{t>0} \int_{\bR^n} \left( G_{3t/4} \ast |\psi_{t/4} \ast a|^2 \right)^{1/2} (x) dx \lesssim \|a\|_{\DotH^{0,1}}.
    \end{equation}
\end{lemma}

\begin{proof}
    Recall that $\DotH^{0,1}$ identifies with the Hardy space $H^1(\bR^n)$. We say a function $a$ on $\bR^n$ is an \emph{$\DotH^{0,1}$-atom} if there exists a ball $B \subset \bR^n$ so that $\supp(a) \subset B$, $\|a\|_{L^\infty} \le |B|^{-1}$, and $\int_{\bR^n} a(x) dx=0$. By atomic decomposition of Hardy spaces (see e.g. \cite[\S III.2.2, Theorem 2]{Stein1993-HA}), it suffices to prove that there is a uniform constant $C>0$ so that for any $a$ as an $\DotH^{0,1}$-atom,
    \begin{equation}
        \label{e:G*psi*atom-bdd}
        \sup_{t>0} \int_{\bR^n} \left( G_{3t/4} \ast |\psi_{t/4} \ast a|^2 \right)^{1/2} (x) dx \le C.
    \end{equation}

    Let $a$ be an $\DotH^{0,1}$-atom and $B$ be the ball so that $\supp(a) \subset B$ and $\|a\|_{L^\infty} \le |B|^{-1}$. By translation, we may assume that $B$ is centered at 0. Note that for any ball $B$ centered at 0, direct computation shows
    \begin{equation}
        \label{e:I_B<G_r(B)^2}
        \frac{\I_B}{|B|}(x) \lesssim G_{r(B)^2}(x), \quad x \in \bR^n.
    \end{equation} 
    We hence get $|a(x)| \le \frac{\I_B}{|B|}(x) \lesssim G_{r(B)^2}(x)$ for any $x \in \bR^n$. Then using \eqref{e:psi_t(x)} and the semigroup property, we obtain that for any $t>0$,
    \begin{align*}
        \int_{\bR^n} \left( G_{3t/4} \ast |\psi_{t/4} \ast a|^2 \right)^{1/2} (x) dx
        &\lesssim \int_{\bR^n} \left( G_{3t/4} \ast |G_{t/4} \ast G_{r(B)^2}|^2 \right)^{1/2} (x) dx \\
        &\lesssim \int_{\bR^n} \left( G_{3t/4} \ast G_{t/4+r(B)^2}^2 \right)^{1/2} (x) dx \\
        &\lesssim \left( \frac{7t+4r(B)^2}{2t+8r(B)^2} \right)^{n/4},
    \end{align*}
    which is uniformly bounded by a constant independent of $t$ and $r(B)$. This proves \eqref{e:G*psi*atom-bdd} and hence completes the proof.
\end{proof}

Now, we provide the proof of Lemma \ref{lemma:A_t}.

\begin{proof}[Proof of Lemma \ref{lemma:A_t}]
    Let $t>0$, $x \in \bR^n$, and $0<s<t/2$. Jensen's inequality yields
    \[ |e^{(t-s)\Delta} \Delta \varphi|^2(x) \le e^{(\frac{3}{4}t-s)\Delta} |e^{\frac{1}{4}t\Delta} \Delta \varphi|^2 (x). \]
    Using \eqref{e:I_B<G_r(B)^2} and the semigroup property of the heat kernel, we get
    \begin{align*}
        \fint_{B(x,s^{1/2})} |e^{(t-s)\Delta} \Delta \varphi(y)|^2 dy 
        &\le \int_{\bR^n} \frac{\I_{B(0,s^{1/2})}}{|B(0,s^{1/2})|} (x-y) e^{(\frac{3}{4}t-s)\Delta} |e^{\frac{1}{4}t\Delta} \Delta \varphi|^2 (y) dy \\
        &\lesssim \int_{\bR^n} G_s(x-y) e^{(\frac{3}{4}t-s)\Delta} |e^{\frac{1}{4}t\Delta} \Delta \varphi|^2 (y) dy \\
        &\lesssim e^{\frac{3}{4}t\Delta} |e^{\frac{1}{4}t\Delta} \Delta \varphi|^2 (x).
    \end{align*}
    Applying this estimate on $\cA_t(\varphi)$ gives
    \begin{equation}
        \label{e:A_t(phi)<G*G*-bd}
        \cA_t(\varphi) \lesssim \int_{\bR^n} \left( t e^{\frac{3}{4}t\Delta} |e^{\frac{1}{4}t\Delta} \Delta \varphi|^2 \right)^{1/2}(x) dx.
    \end{equation}
    
    To prove \eqref{e:A_t(phi)-bd}, since the kernel of $(t^{1/2} e^{t\Delta} \Div)_{t>0}$ satisfies pointwise Gaussian estimates (cf. \eqref{e:psi_t(x)}), we infer from Lemma \ref{lemma:G*psi*a} that
    \[ \cA_t(\varphi) \lesssim \int_{\bR^n} \left( e^{\frac{3}{4}t\Delta} |t^{1/2} e^{\frac{1}{4}t\Delta} \Div \nabla \varphi|^2 \right)^{1/2}(x) dx \lesssim \|\nabla \varphi\|_{\DotH^{0,1}} \eqsim \|\varphi\|_{\DotH^{1,1}}. \]
    The implicit constant is independent of $t$, so \eqref{e:A_t(phi)-bd} follows by taking supremum over $t>0$.

    To prove \eqref{e:A_t(phi)-limit}, by density of $\scrS_\infty$ in $\DotH^{1,1}$, it suffices to prove the limit for $\varphi \in \scrS_\infty$, where in particular, we have $\varphi \in \DotH^{0,1}$ and $\Delta \varphi \in \DotH^{0,1}$. Then \eqref{e:A_t(phi)<G*G*-bd} gives
    \[ \cA_t(\varphi) \lesssim t^{1/2} \int_{\bR^n} \left( G_{3t/4} \ast |G_{t/4} \ast \Delta \varphi|^2 \right)^{1/2}(x) dx \lesssim t^{1/2} \|\Delta \varphi\|_{\DotH^{0,1}}, \]
    which tends to 0 as $t \to 0+$. On the other hand, using the fact that the kernel of $(t \Delta e^{t\Delta})_{t>0}$ satisfies pointwise Gaussian estimates, we get
    \[ \cA_t(\varphi) \lesssim t^{-1/2} \int_{\bR^n} \left( e^{\frac{3}{4}t\Delta} |t\Delta e^{\frac{1}{4}t\Delta} \varphi|^2 \right)^{1/2}(x) dx \lesssim t^{-1/2} \|\varphi\|_{\DotH^{0,1}}, \]
    which tends to 0 as $t \to \infty$. This proves \eqref{e:A_t(phi)-limit}.
\end{proof}

\subsubsection{Estimates of $R_1$}
\label{sssec:R_1}
We first prepare a lemma.
\begin{lemma}
    \label{lemma:f-layer}
    Let $0<\sigma<\tau<\infty$. For any $g \in L^\infty_{-1/2}(\bR^{1+n}_+)$, the function $\tilg(s,y):=g(\sigma+s,y)$ satisfies
    \[ \|\cC^{(2)}_{\tau-\sigma}(\tilg)\|_{L^\infty} \lesssim (\tau-\sigma)^{1/2} \sigma^{-1/2} \|g\|_{L^\infty_{-1/2}(\bR^{1+n}_+)}. \]
\end{lemma}

\begin{proof}
    This lemma is due to \cite[Lemma 7]{Auscher-Dubois-Tchamitchian2004-BMO-1}. For sake of completeness, we provide the computation. Let $B$ be a ball in $\bR^n$ with $r(B)^2<\tau-\sigma$. Then
    \begin{align*}
        \int_0^{r(B)^2} \fint_B |\tilg(s,y)|^2 dsdy
        &= \int_0^{r(B)^2} \fint_B (\sigma+s)^{-1} |(\sigma+s)^{\frac{1}{2}} g(\sigma+s,y)|^2 dsdy \\
        &\quad \le \sigma^{-1} (\tau-\sigma) \|g\|_{L^\infty_{-1/2}(\bR^{1+n}_+)}^2.
    \end{align*}
    Taking supremum on such $B$ gives the desired formula.
\end{proof}

We now provide estimates of $R_1$. Define $\tilg(s,y):=g(s+t/2,y)$. Using change of variable, we have
\[ R_1 = \int_0^{t/2} \int_{\bR^n} |\tilg(s,y)| |e^{(\frac{t}{2}-s)\Delta} a(y)| dsdy. \]
Similarly, we also split $R_1$ into two parts, the main term
\[ \Phi_2 := \int_0^{t/4} \int_{\bR^n} |\tilg(s,y)| |e^{(\frac{t}{2}-s)\Delta} a(y)| dsdy, \]
and the remainder for second iteration
\[ R_2 := \int_{t/4}^{t/2} \int_{\bR^n} |\tilg(s,y)| |e^{(\frac{t}{2}-s)\Delta} a(y)| dsdy. \]
Then the same deduction as $\Phi_1$ (cf. \eqref{e:Phi_1-dual-bd}) implies
\[ \Phi_2 \lesssim \|\cC^{(2)}_{t/4} (\tilg)\|_{L^\infty} \cA_{t/2}(\varphi) \lesssim \left( \frac{2^{-2}}{1-2^{-1}} \right)^{1/2} \|g\|_{L^\infty_{-1/2}(\bR^{1+n}_+)} \cA_{t/2}(\varphi). \]
The last inequality comes from applying Lemma \ref{lemma:f-layer} on $\sigma=(1-2^{-1})t$ and $\tau=(1-2^{-2})t$. Therefore, by iteration, we obtain
\begin{align*}
    R_1 
    &= \sum_{k=2}^\infty \int_0^{2^{-k} t} \int_{\bR^n} \left| g\left( (1-2^{-(k-1)})t+s,y \right) \right| |e^{(2^{-(k-1)} t -s)\Delta} a(y)| dsdy \\
    &\lesssim \sum_{k=2}^\infty 2^{-k/2} \|g\|_{L^\infty_{-1/2}(\bR^{1+n}_+)} \cA_{2^{-k+1} t}(\varphi).
\end{align*}

Combining it with \eqref{e:Phi_1-dual-bd} and \eqref{e:A_t(phi)-bd} gives
\begin{equation}
    \label{e:Phi-bd}
    \begin{aligned}
        \Phi 
        &\lesssim \|\cC^{(2)}_{t/2}(g)\|_{L^\infty} \cA_t(\varphi) + \sum_{k=2}^\infty 2^{-k/2} \|g\|_{L^\infty_{-1/2}(\bR^{1+n}_+)} \cA_{2^{-k+1} t}(\varphi) \\
        &\lesssim (\|g\|_{T^{\infty,2}}+\|g\|_{L^\infty_{-1/2}(\bR^{1+n}_+)}) \|\varphi\|_{\DotH^{1,1}}.
    \end{aligned}
\end{equation}
Moreover, for each $k \ge 1$, Lemma \ref{lemma:A_t} ensures that $\cA_{2^{-k+1} t}(\varphi)$ tends to 0 as $t \to 0$. Since the series $\sum_{k=2}^\infty 2^{-k/2} \|g\|_{L^\infty_{-1/2}(\bR^{1+n}_+)} \|\varphi\|_{\DotH^{1,1}}$ converges, Lebesgue's dominated convergence theorem yields $\Phi$ tends to 0 as $t \to 0+$. This proves \eqref{e:Phi->0}.

Meanwhile, by density of $\scrS_\infty$ in $\DotH^{1,1}$, we get from \eqref{e:Phi-bd} that for any $t>0$ and $g \in L^\infty_{-1/2}(\bR^{1+n}_+) \cap T^{\infty,2}$, $\cL_0(g)(t)$ lies in $\bmo^{-1}$ with
\begin{equation}
    \label{e:L0-Linfty-BMO-1}
    \sup_{t>0} \|\cL_0(g)(t)\|_{\bmo^{-1}} \lesssim \|g\|_{T^{\infty,2}} + \|g\|_{L^\infty_{-1/2}(\bR^{1+n}_+)}.
\end{equation}
Therefore, applying \eqref{e:Phi->0} and a density argument gives that $\cL_0(g)(t)$ tends to 0 in $\bmo^{-1}$ as $t \to 0+$.

\subsection{Continuity at \texorpdfstring{$t>0$}{t>0}}
\label{ssec:L0-t>0}
Next, we consider the continuity at $t>0$. Without loss of generality, we show that $\cL_0(g)(t+h)-\cL_0(g)(t)$ tends to 0 in $\bmo^{-1}$ as $h \to 0+$. The other side follows similarly. Fix $\varphi \in \scrS_\infty$. Fubini's theorem again yields
\begin{align*}
    &|\langle \cL_0(g)(t+h)-\cL_0(g)(t), \varphi \rangle_{L^2(\bR^n)}| \\
    &\quad \le \int_0^t \int_{\bR^n} |g(s,y)| \left|  e^{(t-s)\Delta} (e^{h\Delta} - \bI)\Delta \varphi (y) \right| dsdy \\
    &\qquad + \int_t^{t+h} \int_{\bR^n} |g(s,y)| \left| e^{(t+h-s)\Delta} \Delta \varphi (y) \right| dsdy =: I + II.
\end{align*}

For $I$, using \eqref{e:Phi-bd}, we have
\[ I \lesssim (\|g\|_{T^{\infty,2}} + \|g\|_{L^\infty_{-1/2}(\bR^{1+n}_+)}) \|(e^{h\Delta} - \bI)\varphi\|_{\DotH^{1,1}}. \]
Since the heat semigroup is strongly continuous on $\DotH^{1,1}$, we know that $\|(e^{h\Delta} - \bI)\varphi\|_{\DotH^{1,1}}$ tends to 0 as $h \to 0+$, and so is $I$.

For $II$, using change of variable, one gets
\[ II = \int_0^h \int_{\bR^n} |g(t+s,y)| |e^{(h-s)\Delta} \Delta \varphi(y)| dsdy. \]
Write $\tilg(s,y):=g(t+s,y)$. Lemma \ref{lemma:f-layer} implies that $\|\cC^{(2)}_{h/2}(\tilg)\|_{L^\infty} \lesssim h^{1/2} t^{-1/2} \|g\|_{L^\infty_{-1/2}(\bR^{1+n}_+)}$, and direct computation shows $\|\tilg\|_{L^\infty_{-1/2}(\bR^{1+n}_+)} \lesssim \|g\|_{L^\infty_{-1/2}(\bR^{1+n}_+)}$. So applying \eqref{e:Phi-bd} to $t=h$ and $g=\tilg$, we obtain
\begin{align*}
    II 
    &\lesssim \|\cC^{(2)}_{h/2}(\tilg)\|_{L^\infty} \cA_h(\varphi) + \sum_{k=2}^\infty 2^{-k/2} \|\tilg\|_{L^\infty_{-1/2}(\bR^{1+n}_+)} \cA_{2^{-k+1} h}(\varphi) \\
    &\lesssim \|g\|_{L^\infty_{-1/2}(\bR^{1+n}_+)} \left( h^{1/2} t^{-1/2} \cA_h(\varphi)  + \sum_{k=2}^{\infty} 2^{-k/2} \cA_{2^{-k+1} h}(\varphi) \right),
\end{align*}
which also converges to 0 as $h \to 0+$ by Lebesgue's dominated convergence theorem. A density argument thus yields $\cL_0(g)(t+h)$ converges to $\cL_0(g)(t)$ in $\bmo^{-1}$ as $h \to 0+$.

\subsection{Long-time limit}
\label{ssec:L0-long}
We show that $\cL_0(g)(t)$ tends to 0 in $\bmo^{-1}$ as $t \to \infty$. Pick $\varphi \in \scrS_\infty$. We infer from \eqref{e:Phi-bd} that
\[ |\langle \cL_0(g)(t),\varphi \rangle| \lesssim \|g\|_{T^{\infty,2}} \cA_t(\varphi) + \sum_{k=2}^\infty 2^{-k/2} \|g\|_{L^\infty_{-1/2}(\bR^{1+n}_+)} \cA_{2^{-k+1} t}(\varphi). \]
Using Lemma \ref{lemma:A_t} again, we get for any $k \ge 1$, $\cA_{2^{-k+1} t}(\varphi) \lesssim \|\varphi\|_{\DotH^{1,1}}$ and $\cA_{2^{-k+1} t}(\varphi)$ tends to 0 as $t \to \infty$. Ensured by the convergence of the series $\sum_{k=2}^\infty 2^{-k/2} \|g\|_{L^\infty_{-1/2}(\bR^{1+n}_+)} \|\varphi\|_{\DotH^{1,1}}$, Lebesgue's dominated convergence theorem yields that $\langle \cL_0(g)(t),\varphi \rangle$ tends to 0 as $t \to \infty$. By density, we get $\cL_0(g)$ tends to 0 in $\bmo^{-1}$ as $t \to \infty$. 

This completes the proof.
\section{Proof of Proposition \ref{prop:L1}}
\label{sec:L1}

This section is devoted to proving Proposition \ref{prop:L1}. Again, $\bmo^{-1}$ is endowed with the weak*-topology with respect to $\DotH^{1,1}$.

\subsection{Continuity at \texorpdfstring{$t=0$}{t=0}}
\label{ssec:L1-t=0}
Pick $\varphi \in \scrS_\infty$. Fubini's theorem yields
\[ \langle \cL_1(f)(t),\varphi \rangle_{L^2(\bR^n)} = -\int_0^t \int_{\bR^n} f(s,y) e^{(t+s)\Delta} \bP \nabla \varphi (y) dsdy. \]
Write $b:=\bP \nabla \varphi$ and we get
\begin{align*}
    |\langle \cL_1(f)(t),\varphi \rangle_{L^2(\bR^n)}|
    &\le \int_0^t \int_{\bR^n} |f(s,y)| |e^{s\Delta} b(y)| dsdy \\
    &\quad + \int_0^t \int_{\bR^n} |f(s,y)| |e^{s\Delta} (e^{t\Delta}-\bI)b(y)| dsdy.
\end{align*}
Denote by $I$ and $II$ the first term and the second term on the right-hand side, respectively. 

For $I$, we claim that the function $(s,y) \mapsto |f(s,y)| |e^{s\Delta} b(y)|$ belongs to $L^1(\bR^{1+n}_+)$. If it holds, then Lebesgue's dominated convergence theorem yields that $I$ tends to 0 as $t \to 0+$.

Let us verify the claim. Using duality of tent spaces cf. \eqref{e:dual-T^infty,1}, we get
\begin{equation}
    \label{e:dual-f-b}
    \int_0^\infty \int_{\bR^n} |f(s,y)| |e^{s\Delta} b(y)| dsdy \lesssim \|f\|_{T^{\infty,1}} \| e^{s\Delta} b\|_{T^{1,\infty}}.
\end{equation}
Since $\bP$ is bounded on the Hardy space $H^1(\bR^n) \simeq \DotH^{0,1}$, we know that $b=\bP \nabla \varphi$ lies in $\DotH^{0,1}$. Thanks to \cite{Fefferman-Stein1972-Hp}, we get $e^{s\Delta} b$ lies in $T^{1,\infty}$ with
\begin{equation}
    \label{e:heat-H0,1}
    \|e^{s\Delta} b\|_{T^{1,\infty}} \lesssim \|b\|_{\DotH^{0,1}} \lesssim \|\varphi\|_{\DotH^{1,1}}.
\end{equation}
The claim hence follows. This proves $I$.

For $II$, the above deduction yields
\[ II \lesssim \|f\|_{T^{\infty,1}} \|(e^{t\Delta}-\bI) b\|_{\DotH^{0,1}}. \]
Since the heat semigroup $(e^{t\Delta})$ is strongly continuous on $\DotH^{0,1}$, one gets $\|(e^{t\Delta}-\bI) b\|_{\DotH^{0,1}}$ tends to 0 as $t \to 0+$, and so is $II$. 

In fact, by density of $\scrS_\infty$ in $\DotH^{1,1}$, the above deduction also implies that for any $t>0$ and $f \in T^{\infty,1}$, $\cL_1(f)(t)$ belongs to $\bmo^{-1}$ with
\begin{equation}
    \label{e:L1-Linfty-BMO-1}
    \sup_{t>0} \|\cL_1(f)(t)\|_{\bmo^{-1}} \lesssim \|f\|_{T^{\infty,1}}.
\end{equation}
A density argument shows $\cL_1(f)(t)$ tends to 0 in $\bmo^{-1}$ as $t \to 0+$.

\begin{remark}
    \label{rem:B(u,u)-Linfty-BMO-1}
    Combining \eqref{e:L0-Linfty-BMO-1} and \eqref{e:L1-Linfty-BMO-1}, we rediscover the inequality \eqref{e:B(u,u)-Linfty-BMO-1}. But the argument in \cite[Lemma 8]{Auscher-Dubois-Tchamitchian2004-BMO-1} is much simpler.
\end{remark}

\subsection{Continuity at \texorpdfstring{$t>0$}{t>0}}
\label{ssec:L1-t>0}
Now fix $t>0$. Again, it is enough to prove that $\cL_1(f)(t+h)-\cL_1(f)(t)$ tends to 0 in $\bmo^{-1}$ as $h \to 0+$. Pick $\varphi \in \scrS_\infty$ and write $b=\bP \nabla \varphi$. Fubini's theorem yields
\begin{align*}
    &|\langle \cL_1(f)(t+h)-\cL_1(f)(t), \varphi \rangle_{L^2(\bR^n)}| \\
    &\quad \le \int_0^{t+h} \int_{\bR^n} |f(s,y)| |e^{s\Delta} (e^{(t+h)\Delta}-e^{t\Delta}) b(y)| dsdy \\
    &\qquad + \int_t^{t+h} \int_{\bR^n} |f(s,y)| |e^{(t+s)\Delta} b(y)| dsdy =: I + II.
\end{align*}

For $I$, we infer from \eqref{e:dual-f-b} and \eqref{e:heat-H0,1} that
\[ I \lesssim \|f\|_{T^{\infty,1}} \|(e^{(t+h)\Delta}-e^{t\Delta}) b\|_{\DotH^{0,1}}. \]
Using strong continuity of $(e^{t\Delta})$ on $\DotH^{0,1}$, we get $\|(e^{(t+h)\Delta}-e^{t\Delta}) b\|_{\DotH^{0,1}}$ tends to 0 as $h \to 0+$, and so is $I$.

For $II$, uniform boundedness of $(e^{t\Delta})$ on $\DotH^{0,1}$ implies that the function $(s,y) \mapsto |f(s,y)| |e^{(t+s)\Delta} b(y)|$ belongs to $L^1(\bR^{1+n}_+)$ as
\[ \int_0^\infty \int_{\bR^n} |f(s,y)| |e^{(t+s)\Delta} b(y)| dsdy \lesssim \|f\|_{T^{\infty,1}} \|e^{t\Delta} b\|_{\DotH^{0,1}} \lesssim \|f\|_{T^{\infty,1}} \|b\|_{\DotH^{0,1}}. \]
Lebesgue's dominated convergence theorem hence yields $II$ tends to 0 as $h \to 0+$. A density argument concludes that $\cL_1(f)(t+h)$ converges to $\cL_1(f)(t)$ in $\bmo^{-1}$ as $h \to 0+$. 

\subsection{Long-time limit}
\label{ssec:L1-long}
Let us finish by proving $\cL_1(f)(t)$ tends to 0 in $\bmo^{-1}$ as $t \to \infty$. Pick $\varphi \in \scrS_\infty$ and write $b=\bP \nabla \varphi$. Note that
\begin{align*}
    |\langle \cL_1(f)(t),\varphi \rangle_{L^2(\bR^n)}| 
    &\le \int_0^t \int_{\bR^n} |f(s,y)| |e^{(t+s)\Delta} b(y)| dsdy \\
    &\lesssim \|f\|_{T^{\infty,1}} \|e^{t\Delta} b\|_{\DotH^{0,1}}.
\end{align*}
As $b \in \DotH^{0,1}$, the function $t \mapsto e^{t\Delta} b$ tends to 0 in $\DotH^{0,1}$ as $t \to \infty$, so by density again, we obtain that $\cL_1(f)(t)$ tends to 0 in $\bmo^{-1}$ as $t \to \infty$. 

This completes the proof.

\subsection*{Copyright}
A CC-BY 4.0 \url{https://creativecommons.org/licenses/by/4.0/} public copyright license has been applied by the authors to the present document and will be applied to all subsequent versions up to the Author Accepted Manuscript arising from this submission.

\bibliographystyle{alpha}
\bibliography{sample}

\end{document}